\documentclass[10pt]{article}
\usepackage{geometry}                
\usepackage{graphicx}
\usepackage{fullpage}
\usepackage{amsmath}
\usepackage{amssymb}
\usepackage{amsthm}
\usepackage{url}
\usepackage{epsfig}
\usepackage{color}
\usepackage{refcount}
\usepackage{verbatim}
\usepackage{enumerate}
\usepackage{multirow}
\usepackage{framed}
\usepackage{slashbox}
\usepackage{caption}
\usepackage{subcaption}
\usepackage[normalem]{ulem}

\usepackage[square,numbers,sort&compress]{natbib}

\usepackage{xr}
\usepackage[breaklinks=true,colorlinks,citecolor=blue,linkcolor=blue]{hyperref}

\usepackage{calc}

\usepackage{tikz}
\usetikzlibrary{calc}
\usepackage{epstopdf}
\usetikzlibrary{positioning}
\usepackage{pgffor}

\usepackage{color}
\definecolor{clemson-orange}{RGB}{234,106,32}
\definecolor{chicago-maroon}{RGB}{128,0,0}
\definecolor{cincinnati-red}{RGB}{190,0,0}
\definecolor{soft-cyan}{RGB}{68,85,90}
\usepackage{fullpage}
\usepackage{multicol}

\usepackage[utf8]{inputenc}

\usepackage{array}
\newcolumntype{L}[1]{>{\raggedright\let\newline\\\arraybackslash\hspace{0pt}}m{#1}}
\newcolumntype{C}[1]{>{\centering\let\newline\\\arraybackslash\hspace{0pt}}m{#1}}
\newcolumntype{R}[1]{>{\raggedleft\let\newline\\\arraybackslash\hspace{0pt}}m{#1}}

\newcommand{\bb}{\mathbb}
\newcommand{\R}{\bb R}
\newcommand{\Z}{{\bb Z}}
\newcommand{\N}{{\bb N}}

\theoremstyle{definition}
\newtheorem{theorem}{Theorem}[section]
\newtheorem{lemma}[theorem]{Lemma}
\newtheorem{corollary}[theorem]{Corollary}
\newtheorem{prop}[theorem]{Proposition}

\newtheorem{definition}[theorem]{Definition}

\newtheorem{example}[theorem]{Example}

\DeclareMathOperator*{\conv}{conv}

\DeclareMathOperator*{\poly}{poly}

\newcommand{\D}{\mathcal{D}}

\newcommand{\floor}[1]{\left\lfloor #1 \right\rfloor}
\newcommand{\ceil}[1]{\left\lceil #1 \right\rceil}
\newcommand{\cS}{\mathcal{S}}

\newcommand{\CP}{\mathcal{CP}}

\numberwithin{equation}{section}  

\title{Complexity of branch-and-bound and cutting planes in mixed-integer optimization - II}
\author{Amitabh Basu\thanks{Department of Applied Mathematics and Statistics, Johns Hopkins University, Baltimore, MD, USA ({\tt basu.amitabh@jhu.edu}, {\tt hjiang32@jhu.edu}).} 
\and Michele Conforti\thanks{Dipartimento di Matematica ``Tullio Levi-Civita'', Universit\`a degli Studi Padova, Italy ({\tt conforti@math.unipd.it}, {\tt disumma@math.unipd.it}).}
\and Marco Di Summa\footnotemark[2]
\and Hongyi Jiang\footnotemark[1]}
\date{\today}
\begin{document}

\maketitle

\begin{abstract}
We study the complexity of cutting planes and branching schemes from a theoretical point of view. We give some rigorous underpinnings to the empirically observed phenomenon that combining cutting planes and branching into a branch-and-cut framework can be orders of magnitude more efficient than employing these tools on their own. In particular, we give general conditions under which a cutting plane strategy and a branching scheme give a provably exponential advantage in efficiency when combined into branch-and-cut. The efficiency of these algorithms is evaluated using two concrete measures: number of iterations and sparsity of constraints used in the intermediate linear/convex programs. To the best of our knowledge, our results are the first mathematically rigorous demonstration of the superiority of branch-and-cut over pure cutting planes and pure branch-and-bound.
\end{abstract}

\section{Introduction} In this paper, we consider the following mixed-integer optimization problem:

\begin{equation}\label{eq:main-opt}
\begin{array}{rcll}
\sup\limits & \langle c, x \rangle && \\
\textrm{s.t.} & x &\in & C \cap (\Z^n\times \R^d)
\end{array}
\end{equation}
where $C$ is a closed, convex set in $\R^{n+d}$.

State-of-the-art algorithms for integer optimization are based on two ideas that are at the origin of mixed-integer programming and have been constantly refined: {\em cutting planes} and {\em branch-and-bound}. Decades of theoretical and experimental research into both these techniques is at the heart of the outstanding success of integer programming solvers. Nevertheless, we feel that there is lot of scope for widening and deepening our understanding of these tools. We have recently started building foundations for a rigorous, quantitative theory for analyzing the strengths and weaknesses of cutting planes and branching~\cite{basu-BB-CP}. We continue this project in the current manuscript.

In particular, we provide a theoretical framework to explain an empirically observed phenomenon: algorithms that make a combined use of both cutting planes and branching techniques are more efficient (sometimes by orders of magnitude), compared to their stand alone use in algorithms. We hope that our insights can contribute to a better and more precise understanding of the interaction of cutting planes and branching: which cutting plane schemes and branching schemes complement each other with concrete, provable gains obtained with their combined use, as opposed to not? Not only is a theoretical understanding of this phenomenon lacking, a deeper understanding of the interaction of these methods is considered to be important by both practitioners and theoreticians in the mixed-integer optimization community. To quote an influential computational survey~\cite{lodi2010mixed} ``... it seems that a tighter coordination of the two most fundamental ingredients of the solvers, branching and cutting, can lead to strong improvements."

The main computational burden in any cutting plane or branch-and-bound or branch-and-cut algorithm is the solution of the intermediate convex relaxations. Thus, there are two important aspects to deciding how efficient such an algorithm is: 1) How many linear programs (LPs) or convex optimization problems are solved? 2) How computationally challenging are these convex problems? The first aspect has been widely studied using the concepts of proof size and rank; see~\cite{dash2002exponential,dash2005exponential,dash2010complexity,chvatal1989cutting,chvatal1984cutting,chvatal1980hard,cook2001matrix,bockmayr1999chvatal,eisenbrand2003bounds,rothvoss20130} for a small sample of previous work. Formalizing the second aspect is somewhat tricky and we will focus on a very specific aspect: the {\em sparsity} of the constraints describing the linear program. The collective wisdom of the optimization community says that sparsity of constraints is a highly important aspect in the efficiency of linear programming~\cite{bixby2002solving,eldersveld1992block,reid1982sparsity,vanderbei-sparsity}. Additionally, most successful mixed-integer optimization solvers use sparsity as a criterion for cutting plane selection; see~\cite{dey2015approximating,dey2015some,dey2018analysis} for an innovative line of research. Compared to cutting planes, sparsity considerations have not been as prominent in the choice of branching schemes. This is primarily because for variable disjunctions sparsity is not an issue, and there is relatively less work on more general branching schemes; see~\cite{aardal2000market,pataki2010basis,beame_et_al:LIPIcs:2018:8341,dadush2020complexity,mahmoud2013achieving,ostrowski2008constraint,owen2001experimental,cornuejols2011improved,mahajan2009experiments,karamanov2011branching}. In our analysis, we are careful about the sparsity of the disjunctions as well -- see Definition~\ref{def:sparsity} below.

\subsection{Framework for mathematical analysis.} We now present the formal details of our approach. A {\em cutting plane} for the feasible region of~\eqref{eq:main-opt} is a halfspace $H = \{x \in \R^{n+d} : \langle a, x \rangle \leq \delta\}$ such that $C \cap (\Z^n \times \R^d) \subseteq H$. The most useful cutting planes are those that are not valid for $C$, i.e., $C \not\subseteq H$. There are several procedures used in practice for generating cutting planes, all of which can be formalized by the general notion of a {\em cutting plane paradigm}. A cutting plane paradigm is a function $\CP$ that takes as input any closed, convex set $C$ and outputs a (possibly infinite) family $\CP(C)$ of cutting planes valid for $C \cap (\Z^n \times \R^d)$. Two well-studied examples of cutting plane paradigms are the {\em Chv\'atal-Gomory cutting plane paradigm}~\cite[Chapter 23]{sch} and the {\em split cut paradigm}~\cite[Chapter 5]{conforti2014integer}. We will assume that all cutting planes are rational in this paper. 

State-of-the-art solvers embed cutting planes into a systematic enumeration scheme called {\em branch-and-bound}. The central notion is that of a {\em disjunction}, which is a union of polyhedra $D = Q_1 \cup \ldots \cup Q_k$ such that $\Z^n \times \R^d \subseteq D$, i.e., the polyhedra together cover all of $\Z^n \times \R^d$. One typically uses a (possibly infinite) family of disjunctions for potential deployment in algorithms. A well-known example is the family of {\em split disjunctions} that are of the form $D_{\pi,\pi_0}:= \{x\in \R^{n+d} : \langle \pi, x \rangle \leq \pi_0\} \cup \{x\in \R^{n+d} : \langle \pi, x \rangle \geq \pi_0+1\}$, where $\pi \in \Z^n \times \{0\}^d$ and $\pi_0\in \Z$. When the first $n$ coordinates of $\pi$ correspond to a standard unit vector, we get {\em variable disjunctions}, i.e., disjunctions of the form $\{x: x_i \leq \pi_0\} \cup \{x : x_i \geq \pi_0+1\}$, for $i=1, \ldots, n$.

A family of disjunctions $\mathcal{D}$ can also form the basis of a cutting plane paradigm. Given any disjunction $D$, any halfspace $H$ such that $C \cap D \subseteq H$ is a cutting plane, since $C \cap (\Z^n \times \R^d) \subseteq C \cap D$ by definition of a disjunction. The corresponding cutting plane paradigm $\CP(C)$, called {\em disjunctive cuts based on $\D$}, is the family of all such cutting planes derived from disjunctions in $\D$. Two well-known examples are the family of {\em split cuts}, based on the family of split disjunctions defined above, and the family of {\em lift-and-project cuts} derived from variable disjunctions.

In the following we assume that all convex optimization problems that need to be solved have an optimal solution or are infeasible.

\begin{definition}\label{def:BC-algorithm} A {\em branch-and-cut algorithm based on a family $\mathcal{D}$ of disjunctions and a cutting plane paradigm $\CP$} maintains a list $\mathcal{L}$ of {\em convex subsets} of the initial set $C$ which are guaranteed to contain the optimal point, and a lower bound $LB$ that stores the objective value of the best feasible solution found so far (with $LB=-\infty$ if no feasible solution has been found). At every iteration, the algorithm selects one of these subsets $N \in \mathcal{L}$ and solves the convex optimization problem $\sup\{ \langle c, x\rangle: x \in N\}$ to obtain $x^N$. If the objective value is less than or equal to $LB$, then this set $N$ is discarded from the list $\mathcal{L}$. Else, if $x^N$ satisfies the integrality constraints, $LB$ is updated with the value of $x^N$ and $N$ is discarded from the list. Otherwise, the algorithm makes a decision whether to {\em branch} or to {\em cut}. In the former case, a disjunction $D = (Q_1 \cup \ldots \cup Q_k) \in \mathcal{D}$ is chosen such that $x^N \not\in D$ and the list is updated $\mathcal{L} := \mathcal{L} \setminus \{N\} \cup \{Q_1\cap N, \ldots, Q_k\cap N\}$. If the decision is to cut, then the algorithm selects a cutting plane $H\in \CP(N)$ such that $x^N \not\in H$, and updates the relaxation $N$ by adding the cut $H$, i.e., updates $\mathcal{L} := \mathcal{L} \setminus \{N\} \cup \{N\cap H\}$.
 \end{definition}

Motivated by the above, we will refer to a family $\D$ of disjunctions also as a {\em branching scheme}. In a branch-and-cut algorithm, if one always chooses to add a cutting plane and never uses a disjunction to branch, then it is said to be a {\em (pure) cutting plane algorithm} and if one does not use any cutting planes ever, then it is called a {\em (pure) branch-and-bound algorithm}. We note here that in practice, when a decision to cut is made, several cutting planes are usually added as opposed to just one single cutting plane like in Definition~\ref{def:BC-algorithm}. In our mathematical framework, allowing only a single cut makes for a seamless generalization from pure cutting plane algorithms, and also makes quantitative analysis easier.

\begin{definition}\label{def:BC-tree} The execution of any branch-and-cut algorithm on a mixed-integer optimization instance can be represented by a tree. Every convex relaxation $N$ processed by the algorithm is denoted by a node in the tree. If the optimal value for $N$ is not better than the current lower bound, or is integral, $N$ is a leaf. Otherwise, in the case of a branching, its children are $Q_1\cap N, \ldots, Q_k\cap N$, and in the case of a cutting plane, there is a single child representing $N \cap H$ (we use the same notation as in Definition~\ref{def:BC-algorithm}). This tree is called the {\em branch-and-cut tree} ({\em branch-and-bound tree}, if no cutting planes are used). If no branching is done, this tree (which is really a path) is called a {\em cutting plane proof}. The {\em size} of the tree or proof is the total number of nodes.
\end{definition}

\paragraph{Proof versus algorithm.} Although we use the word ``algorithm" in Definition~\ref{def:BC-algorithm}, it is technically a {\em non-deterministic} algorithm, or equivalently, a proof schema or proof system for optimality~\cite{arora2009computational} (leaving aside the question of finite termination for now). This is because no indication is given on how the important decisions are made: Which set $N$ to process from $\mathcal{L}$? Branch or cut? Which disjunction or cutting plane to use? If these are made concrete, one would obtain a standard deterministic algorithm (assuming, for the moment, finite termination on all instances). Nevertheless, the proof system is very useful for obtaining information theoretic lower bounds on the efficiency of any deterministic branch-and-cut algorithm. Moreover, one can prove the validity of any upper bound on the objective, i.e., the validity of $\langle c, x \rangle \leq \gamma$ by exhibiting a branch-and-cut tree where this inequality is valid for all the leaves. If $\gamma$ is the optimal value, this is a proof of optimality, but one may often be interested in the branch-and-cut/branch-and-bound/cutting plane proof complexity of other valid inequalities as well. The connections between integer programming and proof complexity has a long history;  see~\cite{beame_et_al:LIPIcs:2018:8341,dadush2020complexity,bonet1997lower,razborov2017width,impagliazzo1994upper,buss1996cutting,cook1987complexity,goerdt1990cutting,goerdt1991cutting,clote1992cutting,pudlak1997lower,pudlak1999complexity,krajivcek1998discretely,kojevnikov2007improved,grigoriev2002complexity}, to cite a few. Our results can be interpreted in the language of proof complexity as well.

Another subtlety to keep in mind is that one could add to the power of such a branch-and-cut proof system by relaxing the requirement that the current optimal solution $x^N$ should be eliminated by the chosen disjunction or cutting plane. This can make a difference -- an instance may have a finite proof in the strengthened system while no finite proof exists in the original system~\cite{owen2001disjunctive}. When required, we will use the phrase {\em restricted proof} to refer to a proof that imposes the restriction of eliminating $x^N$ at every node $N$ of the proof tree.
\bigskip

Recall that we quantify the complexity of any branch-and-bound/cutting plane/branch-and-cut algorithm using two aspects: the number of LP relaxations processed and the sparsity of the constraints defining the LPs. The number of LP relaxations processed is given precisely by the number of nodes in the corresponding tree (Definition~\ref{def:BC-tree}). Sparsity is formalized in the following definitions.

\begin{definition}\label{def:sparsity} Let $1 \leq s \leq n+d$ be a natural number that we call the {\em sparsity parameter}. Then the pair $(\CP, s)$ will denote the restriction of the paradigm $\CP$ that only reports the sub-family of cutting planes that can be represented by inequalities with at most $s$ non-zero coefficients; the notation $(\CP,s)(C)$ will be used to denote this sub-family for any particular convex set $C$. Similarly, $(\D,s)$ will denote the sub-family of the family of disjunctions $\D$ such that each polyhedron in the disjunction has an inequality description where every inequality has at most $s$ non-zero coefficients.\end{definition}

Cutting plane proof systems with restrictions on the ``depth" of the cutting planes have been considered in the proof complexity literature; see~\cite{goerdt1991cutting,hirsch2005simulating}.

\subsection{Our Results}

\subsubsection{Sparsity versus size.} Our first set of results considers the trade-off between the sparsity parameter $s$ and the number of LPs processed, i.e., the size of the tree. There are several avenues to explore in this direction. For example, one could compare pure branch-and-bound algorithms based on $(\D, s_1)$ and $(\D, s_2)$, i.e., fix a particular disjunction family $\D$ and consider the effect of sparsity on the branch-and-bound tree sizes. One could also look at two different families of disjunctions $\D_1$ and $\D_2$ and look at their relative tree sizes as one turns the knob on the sparsity parameter. Similar questions could be asked about cutting plane paradigms $(\CP_1, s_1)$ and $(\CP_2, s_2)$ for interesting paradigms $\CP_1, \CP_2$. Even more interestingly, one could compare pure branch-and-bound and pure cutting plane algorithms against each other.

We first focus on pure branch-and-bound algorithms based on the family $\cS$ of split disjunctions. A very well-known example of pure integer instances (i.e., $d=0$) due to Jeroslow~\cite{Jeroslow1974} shows that if the sparsity of the splits used is restricted to be 1, i.e., one uses only variable disjunctions, then the branch-and-bound algorithm will generate an exponential (in the dimension $n$) sized tree. On the other hand, if one allows fully dense splits, i.e., sparsity is $n$, then there is a tree with just 3 nodes (one root, and two leaves) that solves the problem. We ask what happens in Jeroslow's example if one uses split disjunctions with sparsity $s > 1$. Our first result shows that unless the sparsity parameter $s = \Omega(n)$, one cannot get constant size trees, and if the sparsity parameter $s = O(1)$, then the tree is of exponential size.

\begin{theorem}\label{thm::exp-Jeroslow-ineq}
Let $H$ be the halfspace defined by inequality $2\sum_{i=1}^nx_i\leq n$, where $n$ is an odd number. Consider the instances of (\ref{eq:main-opt}) with $d=0$, the objective $\sum_{i=1}^nx_i$ and $C=H \cap [0,1]^n$. The optimum is $\floor{\frac{n}{2}}$, and any branch-and-bound proof with sparsity $s\leq \floor{\frac{n}{2}}$ that certifies $\sum_{i=1}^nx_i\leq \floor{\frac{n}{2}}$ has size at least  $\Omega(2^{\frac{n}{2s}})$.
\end{theorem}

The above instance is a modification of Jeroslow's instance; Jeroslow's instance uses an equality constraint instead of an inequality. However, the same argument applies for Jeroslow's instance.

\begin{corollary}\label{cor:exp-Jeroslow-eq}
Let $H$ be the hyperplane defined by equality $2\sum_{i=1}^nx_i= n$, where $n$ is an odd number. Consider the instances of (\ref{eq:main-opt}) with $d=0$, the objective $\sum_{i=1}^nx_i$ and $C=H\cap [0,1]^n$. This problem is infeasible, and any branch-and-bound proof of infeasibility with sparsity $s\leq \floor{\frac{n}{2}}$ has size at least $\Omega(2^{\frac{n}{2s}})$.
\end{corollary}

The bounds in Theorem~\ref{thm::exp-Jeroslow-ineq} give a constant lower bound when $s = \Omega(n)$. We establish another lower bound which does better in this regime.

\begin{theorem}\label{thm::exp-Jeroslow-ineq-II}
Let $H$ be the halfspace defined by inequality $2\sum_{i=1}^nx_i\leq n$, where $n$ is an odd number. Consider the instances of (\ref{eq:main-opt}) with $d=0$, the objective $\sum_{i=1}^nx_i$ and $C=H \cap [0,1]^n$. The optimum is $\floor{\frac{n}{2}}$, and any branch-and-bound proof with sparsity $s\leq \floor{\frac{n}{2}}$ that certifies $\sum_{i=1}^nx_i\leq \floor{\frac{n}{2}}$ has size at least $\Omega\left(\sqrt{\frac{n(n-s)}{s}}\right)$.\end{theorem}

Next we consider the relative strength of cutting planes and branch-and-bound. Our previous work has studied conditions under which one method can dominate the other, depending on which cutting plane paradigm and branching scheme one chooses~\cite{basu-BB-CP}. For this paper, the following result from~\cite{basu-BB-CP} is relevant: for every convex 0/1 pure integer instance, any branch-and-bound proof based on variable disjunctions can be ``simulated" by a lift-and-project cutting plane proof without increasing the size of the proof (versions of this result for {\em linear} 0/1 programming were known earlier; see~\cite{dash2002exponential,dash2005exponential}). Moreover, in~\cite{basu-BB-CP} we constructed a family of stable set instances where lift-and-project cuts give exponentially shorter proofs than branch-and-bound. This is interesting because lift-and-project cuts are disjunctive cuts based on the same family of variable disjunctions, so it is not a priori clear that they have an advantage. These results were obtained with no regard for sparsity. We now show that once we also track the sparsity parameter, this advantage can disappear.

\begin{theorem}\label{thm:BB-vs-CP-sparsity}
Let $H$ be the halfspace defined by inequality $2\sum_{i=1}^nx_i\leq n$, where $n$ is an odd number. Consider the intances of (\ref{eq:main-opt}) with $d=0$, the objective $\sum_{i=1}^{\ceil{\frac{n}{2}}}x_i$ and $C=H\cap[0,1]^n$. The optimum is $\floor{\frac{n}{2}}$, and there is a branch-and-bound algorithm based on variable disjunctions, i.e., the family of split disjunctions with sparsity $1$, that certifies $\sum_{i=1}^{\ceil{\frac{n}{2}}}x_i\leq \floor{\frac{n}{2}}$ in $O(n)$ steps. However, any cutting plane for $C$ with sparsity $s\leq \floor{\frac{n}{2}}$ is trivial, i.e., valid for $[0,1]^n$, no matter what cutting plane paradigm is used to derive it.
\end{theorem}

\subsubsection{Superiority of branch-and-cut.} We next consider the question of when combining branching and cutting planes is {\em provably} advantageous. For this question, we leave aside the complications arising due to sparsity considerations and focus only on the size of proofs. The following discussion and results can be extended to handle the issue of sparsity as well, but we leave it out of this extended abstract. 

Given a cutting plane paradigm $\CP$, and a branching scheme $\D$, are there families of instances where branch-and-cut based on $\CP$ and $\D$ does provably better than pure cutting planes based on $\CP$ alone and pure branch-and-bound based on $\D$ alone? If a cutting plane paradigm $\CP$ and a branching scheme $\D$ are such that either for every instance, $\CP$ gives cutting plane proofs of size at most a polynomial factor larger than the shortest branch-and-bound proofs with $\D$, or vice versa, for every instance $\D$ gives proofs of size at most polynomially larger than the shortest cutting plane proofs based on $\CP$, then combining them into branch-and-cut is likely to give no substantial improvement since one method can always do the job of the other, up to polynomial factors. As mentioned above, prior work \cite{basu-BB-CP} had shown that disjunctive cuts based on variable disjunctions (with no restriction on sparsity) dominate branch-and-bound based on variable disjunctions for pure 0/1 instances, and as a consequence branch-and-cut based on these paradigms is dominated by pure cutting planes. In the next theorem, we show that the situation completely reverses if one considers a broader family of disjunctions (still restricted to the pure integer case).

\begin{theorem}\label{thm::general-BB>=CP}
Let $C\subseteq \R^n$ be a closed, convex set. Let $k\in \N$ be a fixed natural number and let $\D$ be any family of disjunctions that contains all split disjunctions, such that all disjunctions in $\D$ have at most $k$ terms in the disjunction. If a valid inequality $\langle c,x\rangle\leq \delta$ for $C\cap \Z^n$ has a cutting plane proof of size $L$ using disjunctive cuts based on $\D$, then there exists a branch-and-bound proof of size at most $(k+1)L$ based on $\D$. Moreover, there is a family of instances where branch-and-bound based on split disjunctions solves the problem in $O(1)$ time whereas there is a polynomial lower bound on split cut proofs.
\end{theorem}

A consequence of Theorem \ref{thm::general-BB>=CP} is that any cutting plane proof based on Chv\'atal-Gomory cuts can be replaced by a branch-and-bound proof based on split disjunctions with a constant blow up in size (since Chv\'atal-Gomory cuts are a subset of split cuts). This special case was also proved in earlier work by Beame et al.~\cite[Theorem 12]{beame_et_al:LIPIcs:2018:8341}. We also emphasize that the proof of Theorem~\ref{thm::general-BB>=CP} crucially uses the fact that we have a class of disjunctions that is rich enough to include {\em all} split disjunctions.

With similar analysis as Theorem \ref{thm::general-BB>=CP}, we can get the following theorem that takes sparsity into account as well.
\begin{theorem}
Let $C\in \R^n$ be a closed, convex set. Let $\langle c,x\rangle\leq \delta$ be a valid inequality for $C\cap \Z^n$. If there exists a cutting plane proof of size $L$ and sparsity $s$ certifying the validity of this inequality, which is derived using general split disjunctions of sparsity $s$, then there exists a branch-and-bound proof of sparsity $s$ which proves the validity and takes at most $O(L)$ iterations.
\end{theorem}

The above discussion and theorem motivate the following definition which formalizes the situation where no method dominates the other. To make things precise, we assume that there is a well-defined way to assign a concrete {\em size} to any instance of~\eqref{eq:main-opt}; see~\cite{GroetschelLovaszSchrijver-Book88} for a discussion on how to make this formal. Additionally, when we speak of an instance, we allow the possibility of proving the validity of any inequality valid for $C \cap (\Z^n \times \R^d)$, not necessarily related to an upper bound on the objective value. Thus, an instance is a tuple $(C, c, \gamma)$ such that $\langle c, x \rangle \leq \gamma$ for all $x \in C \cap (\Z^n \times \R^d)$.

\begin{definition}\label{complementary-pair} A cutting plane paradigm $\CP$ and a branching scheme $\D$ are {\em complementary} if there is a family of instances where $\CP$ gives polynomial (in the size of the instances) size proofs and the shortest branch-and-bound proof based on $\D$ is exponential (in the size of the instances), and there is another family of instances where $\D$ gives polynomial size proofs while $\CP$ gives exponential size proofs. \end{definition}

We wish to formalize the intuition that branch-and-cut is expected to be exponentially better than branch-and-bound or cutting planes alone for complementary pairs of branching schemes and cutting plane paradigms. But we need to make some mild assumptions about the branching schemes and cutting plane paradigms. {\em All known branching schemes and cutting plane methods from the literature satisfy these conditions}.

\begin{definition}\label{def:regular} A branching scheme is said to be {\em regular} if no disjunction involves a continuous variable, i.e., each polyhedron in the disjunction is described using inequalities that involve only the integer constrained variables.

A branching scheme $\D$ is said to be {\em embedding closed} if disjunctions from higher dimensions can be applied to lower dimensions. More formally, let $n_1$, $n_2$, $d_1$, $d_2 \in \N$. If $D \in \D$ is a disjunction in $\R^{n_1} \times \R^{d_1} \times  \R^{n_2} \times \R^{d_2}$ with respect to $\Z^{n_1} \times \R^{d_1} \times \Z^{n_2} \times \R^{d_2}$, then the disjunction $D\cap (\R^{n_1} \times \R^{d_1} \times \{0\}^{n_2} \times \{0\}^{d_2})$, interpreted as a set in $\R^{n_1} \times \R^{d_1}$, is also in $\D$ for the space $\R^{n_1}\times \R^{d_1}$ with respect to $\Z^{n_1} \times \R^{d_1}$ (note that $D\cap (\R^{n_1} \times \R^{d_1} \times \{0\}^{n_2} \times \{0\}^{d_2})$, interpreted as a set in $\R^{n_1} \times \R^{d_1}$, is certainly a disjunction with respect to $\Z^{n_1} \times \R^{d_1}$; we want $\D$ to be closed with respect to such restrictions).

A cutting plane paradigm $\CP$ is said to be {\em regular} if it has the following property, which says that adding ``dummy variables" to the formulation of the instance should not change the power of the paradigm. Formally, let $C \subseteq \R^{n} \times \R^{d}$ be any closed, convex set and let $C' = \{(x,t) \in \R^{n}\times \R^{d} \times \R: x \in C,\;\; t = \langle f, x \rangle\}$ for some $f \in \R^n$. Then if a cutting plane $\langle a, x \rangle \leq b$ is derived by $\CP$ applied to $C$, i.e., this inequality is in $\CP(C)$, then it should also be in $\CP(C')$, and conversely, if $\langle a, x\rangle + \mu t \leq b$ is in $\CP(C')$, then the equivalent inequality $\langle a + \mu f, x \rangle \leq b$ should be in $\CP(C)$.

A cutting plane paradigm $\CP$ is said to be {\em embedding closed} if cutting planes from higher dimensions can be applied to lower dimensions. More formally, let $n_1, n_2, d_1, d_2 \in \N$. Let $C \subseteq \R^{n_1} \times \R^{d_1}$ be any closed, convex set. If the inequality $\langle c_1, x_1\rangle + \langle a_1, y_1\rangle + \langle c_2, x_2\rangle + \langle a_2, y_2\rangle \leq \gamma$ is a cutting plane for $C \times \{0\}^{n_2} \times \{0\}^{d_2}$ with respect to $\Z^{n_1} \times \R^{d_1} \times \Z^{n_2} \times \R^{d_2}$ that can be derived by applying $\CP$ to $C \times \{0\}^{n_2} \times \{0\}^{d_2}$, then the cutting plane $\langle c_1, x_1\rangle+ \langle a_1, y_1\rangle \leq \gamma$ that is valid for $C \cap (\Z^{n_1} \times \R^{d_1})$ should also belong to $\CP(C)$.

A cutting plane paradigm $\CP$ is said to be {\em inclusion closed}, if for any two closed convex sets $C \subseteq C'$, we have $\CP(C') \subseteq \CP(C)$. In other words, any cutting plane derived for $C'$ can also be derived for a subset $C$.
\end{definition}

\begin{theorem}\label{thm:BC>BB/CP} Let $\D$ be a regular, embedding closed branching scheme and let $\CP$ be a regular, embedding closed, and inclusion closed cutting plane paradigm such that $\D$ includes all variable disjunctions and $\CP$ and $\D$ form a complementary pair. Then there exists a family of instances of~\eqref{eq:main-opt} which have polynomial size branch-and-cut proofs, whereas any branch-and-bound proof based on $\D$ and any cutting plane proof based on $\CP$ is of exponential size.
\end{theorem}

\begin{example}\label{ex:complementary} As a concrete example of a complementary pair that satisfies the other conditions of Theorem~\ref{thm:BC>BB/CP}, consider $\CP$ to be the Chv\'atal-Gomory paradigm and $\D$ to be the family of variable disjunctions. From their definitions, they are both regular and $\D$ is embedding closed. The Chv\'atal-Gomory paradigm is also embedding closed and inclusion closed. For the Jeroslow instances from Theorem~\ref{thm::exp-Jeroslow-ineq}, the single Chv\'atal-Gomory cut $\sum_{i=1}^n x_i \leq \lfloor \frac{n}{2} \rfloor$ proves optimality, whereas variable disjunctions produce a tree of size $2^{\lfloor \frac{n}{2}\rfloor}$. On the other hand, consider the set $T$, where $T = \conv\{(0,0), (1,0), (\frac12,h)\}$ and the valid inequality $x_2 \leq 0$ for $T \cap \Z^2$. Any Chv\'atal-Gomory paradigm based proof has size exponential in the size of the input, i.e., every proof has length at least $\Omega(h)$~\cite{sch}. On the other hand, a single disjunction on the variable $x_1$ solves the problem.

In~\cite{basu-BB-CP}, we also studied examples of disjunction families $\D$ such that disjunctive cuts based on  $\D$ are complementary to branching schemes based on $\D$.\end{example}

Example~\ref{ex:complementary} shows that the classical Chv\'atal-Gomory cuts and variable branching are complementary and thus give rise to a superior branch-and-cut routine when combined by Theorem~\ref{thm:BC>BB/CP}. As discussed above, for 0/1 problems, lift-and-project cuts and variable branching do {\em not} form a complementary pair, and neither do split cuts and split disjunctions by Theorem~\ref{thm::general-BB>=CP}. It would be nice to establish the converse of Theorem~\ref{thm:BC>BB/CP}: if there is a family where branch-and-cut is exponentially superior, then the cutting plane paradigm and branching scheme are complementary. In Theorem~\ref{thm:conv-BC} below, we prove a partial converse along these lines in the pure integer setting. This partial converse requires the disjunction family to include all split disjunctions. It would be more satisfactory to establish similar results without this assumption. More generally, it remains an open question if our definition of complementarity is an exact characterization of when branch-and-cut is superior.

\begin{theorem}\label{thm:conv-BC} Let $\D$ be a branching scheme that includes all split disjunctions and let $\CP$ be any cutting plane paradigm. Suppose that for every pure integer instance and any cutting plane proof based on $\CP$ for this instance, there is a branch-and-bound proof based on $\D$ of size at most a polynomial factor (in the size of the instance) larger. Then for any branch-and-cut proof based on $\D$ and $\CP$ for a pure integer instance, there exists a pure branch-and-bound proof based on $\D$ that has size at most polynomially larger than the branch-and-cut proof.
\end{theorem}

The high level message that we extract from our results is the formalization of the following simple intuition. For branch-and-cut to be superior to pure cutting planes or pure branch-and-bound, one needs the cutting planes and branching scheme to do ``sufficiently different" things. For example, if they are both based on the same family of disjunctions (such as lift-and-project cuts and variable branching, or the setting of Theorem~\ref{thm::general-BB>=CP}), then we do not get any improvements with branch-and-cut. The definition of a complementary pair attempts to make the notion of ``sufficiently different" formal and Theorem~\ref{thm:BC>BB/CP} derives the concrete superior performance of branch-and-cut from this formalization.

\section{Proofs}

\subsection{Proof of Theorem~\ref{thm::exp-Jeroslow-ineq}}

We first give necessary definitions and prove a lemma.

\begin{definition}\label{def::gen-var}
Consider the instances in Theorem \ref{thm::exp-Jeroslow-ineq}, and the branch-and-bound tree $T$  produced by split disjunctions to solve it. Assume node $N$ of $T$ contains at least one integer point in $\{0,1\}^n$, and $D_1,D_2,\ldots, D_r$ are the split disjunctions used to derive $N$ from the root of $T$. For $1\leq  j\leq r$, $D_j$ is a {\em true split disjunction of $N$} if both of the two halfspaces of $D_j$ have a nonempty intersection with the integer hull of the corresponding parent node, i.e. the parent node's integer hull is split into two nonempty parts by $D_j$. Otherwise, it is called a {\em false split disjunction of $N$}. We define the {\em generation variable set of $N$} as the index set $I\subseteq \{1,2,\ldots, n\}$ such that it consists of all the indices of the variables involved in the true split disjunctions of $N$. The generation set of the root node is empty.
\end{definition}

\begin{lemma}\label{lem::true-split}
Consider the instances in Theorem \ref{thm::exp-Jeroslow-ineq}, and the branch-and-bound tree $T$  produced by split disjunctions with sparsity parameter $s<\floor{\frac{n}{2}}$ to solve it. For any node $N$ of $T$ with at least one feasible integer point $v=(v_1,v_2,\ldots,v_n)\in \{0,1\}^n$, let $P$, $P_I$ and $I$ denote the relaxation, the integer hull and the generation variable set corresponding to $N$. Define $V:=\{(x_1,x_2,\ldots,x_n)\in \{0,1\}^n:x_i=v_i \mbox{ for }i\in I, \sum_{j=1}^nx_i= \floor{\frac{n}{2}}\}$.

If $\lvert I\rvert\leq \floor{\frac{n}{2}}-s$, then we have:

\begin{enumerate}
\item[(i)] $V\neq \emptyset$ and $V\subseteq P_I\cap \{0,1\}^n$;
\item[(ii)] the objective LP value of $N$ is $\frac{n}{2}$.
\end{enumerate}
\end{lemma}

\begin{proof}
We first give a proof of (i). Since $v$ is a feasible integer point, $0\leq \sum_{i\in I}v_i\leq  \sum_{i=1}^n v_i \leq\floor{\frac{n}{2}}$. Thus, there exists $v'= (v'_1,v'_2,\ldots,v'_n)$, where $v_i'=v_i$ for $i\in I$ and $\sum_{i=1}^nv_i'=\floor{\frac{n}{2}}$. So $v' \in V\neq \emptyset$.

For each $v^*\in V$, we wish to show that $v^* \in P$. This will show that $v^*\in P_I$ and $V\subseteq P_I$.  Consider any inequality describing $P$; if it is not the original defining inequality $\sum_{i=1}^nx_i\leq \frac{n}{2}$ or a 0/1 bound on a variable, then this inequality was introduced on the path from the root to $N$. A false split disjunction cannot remove $v^*$ since $v^*$ is integral. Consider an inequality coming from a true split disjunction. Let $\sum_{i\in S}a_ix_i\leq \delta^*$ for some $S\subseteq I$ be such an inequality. Since $v \in P_I$ and $v^*_i = v_i$ for $i\in I$, we observe that  $\sum_{i\in S}a_iv_i = \sum_{i\in S}a_iv^*_i\leq \delta^*$.

We will prove (ii) by contradiction, so we assume the objective LP value of $N$ is strictly less than $\frac{n}{2}$. Let $P_0$ denote the relaxation corresponding to the root node.
 Assume $\ell\in \{1,2,\ldots,n\}\backslash I$.

Since $\lvert I\rvert\leq \floor{\frac{n}{2}}-s$, there exists $v^1=(v_1^1,v_2^1,\ldots,v_n^1)\in V$, where $v_\ell^1=0$. Define $v^2=(v_1^2,v_2^2,\ldots, v_n^2)$, where $v_\ell^2=\frac{1}{2}$, and $v_i^2=v_i^1$ for $i\in\{1,2,\ldots,n\}\backslash \{\ell\}$. It is clear that $v^2\in P_0$, and $v^2\notin P$ since the LP value is assumed to be strictly less than $\frac{n}{2}$. Since $\ell\notin I$, there must be a halfspace $\hat H$ coming from a false split disjunction of $N$ that excludes $v^2$. The inequality describing this halfspace $\hat H$ must involve variable $x_\ell$, otherwise $v^1$ also violates $\hat H$, which leads to a contradiction since $\hat H$ comes from a false split disjunction and therefore cannot cut off any integer point. Hence assume the inequality describing $\hat H$ is $a_\ell x_\ell+\sum_{i\in S}a_ix_i\leq \delta$ for some $S\subseteq \{1,2,\ldots,n\}\backslash\{\ell\}$, and $\lvert S\rvert\leq s-1$ (since the sparsity of the disjunctions is restricted to be at most $s$). Since $\sum_{i\in I}v^1_i\leq \floor{\frac{n}{2}}-s$, we have $\sum_{i\notin I\cup \{\ell\}}v^1_i\geq s$, and there exists $r\in \{1,2,\ldots,n\}\backslash (S\cup I\cup\{\ell\})$ such that $v_r^1=1$. Let $v^3=(v_1^3,v_2^3,\ldots,v_n^3)$, where $v_\ell^3=1$, $v_r^3=0$, and $v_i^3=v_i^1$ for $i\neq \ell,r$. By definition of $V$, $v^3\in V$. Since $v^1, v^3$ are integral, and $\hat H$ comes from a false split disjunction, $\hat H$ must be valid for $v^1$ and $v^3$. Thus, we have
\begin{align}
    &a_\ell\cdot 0+\sum_{i\in S}a_iv_i^1=a_\ell\cdot 0+\sum_{i\in S}a_iv_i^2\leq \delta\label{eq::v1},\\
    &a_\ell\cdot 1+\sum_{i\in S}a_iv_i^3=a_\ell\cdot 1+\sum_{i\in S}a_iv_i^1=a_\ell\cdot 1+\sum_{i\in S}a_iv_i^2\leq \delta\label{eq::v3}.
\end{align}
Summing up (\ref{eq::v1}) and (\ref{eq::v3}) and dividing by 2, we get
\begin{equation}
    a_\ell\cdot \frac{1}{2}+\sum_{i\in S}a_iv_i^2=a_\ell\cdot v^2_\ell+\sum_{i\in S}a_iv_i^2\leq \delta,
\end{equation}
which implies that $\hat H$ is valid for $v^2$. This is a contradiction.
\end{proof}

\begin{proof}[Proof of Theorem \ref{thm::exp-Jeroslow-ineq}]
For a node $N$ of the branch-and-bound tree containing at least one integer point, if it is derived by exactly $m$ true split disjunctions, then we say it is a node of generation $m$. By Lemma \ref{lem::true-split}, if $m\leq \frac{1}{s}\big\lfloor\frac{n}{2}\big\rfloor-1$, then a node $N$ of generation $m$ has LP objective value $\frac{n}{2}$, and in the subtree rooted at $N$ there must exist at least two descendants from generation $m+1$, since the leaf nodes must have LP values less than or equal to $\lfloor \frac{n}{2} \rfloor$. Therefore, there are at least $2^m$ nodes of generation $m$ when $m\leq \frac{1}{s}\big\lfloor\frac{n}{2}\big\rfloor-1$. This finishes the proof.
\end{proof}

\subsection{Proof of Theorem \ref{thm::exp-Jeroslow-ineq-II}}

\begin{lemma}\label{lem:sperner} Let $w_1, \ldots, w_k \in \Z\setminus\{0\}$ and $W \in \Z$. Then the number of 0/1 solutions to $\sum_{j=1}^k w_jx_j = W$ is at most ${k\choose \lfloor k/2\rfloor}$. \end{lemma}

\begin{proof} Let $P:= \{i \in \{1, \ldots, k\}: w_i > 0\}$ and $N:= \{i \in \{1, \ldots, k\}: w_i < 0\}$. By making the variable change $x_i = 1 - y_i$ for $i\in N$ and $x_i = y_i$ for $i\in P$, it is seen that the number of 0/1 solutions to $\sum_{i=1}^k w_ix_i = W$ is the same as the number of 0/1 solutions to $\sum_{i\in P} w_i y_i + \sum_{i\in N}(-w_i)y_i = W - \sum_{i\in N}w_i$. Writing this a bit more cleanly, we want to upper bound the number of 0/1 solutions to $\sum_{i=1}^k w'_iy_i = W'$, where $w_i' > 0$ for all $i \in \{1, \ldots, k\}$ and $W' \in \Z$. The collection of subsets $I \subseteq \{1, \ldots, k\}$ that are solutions to $\sum_{i=1}^k w'_iy_i = W'$ is an antichain in the lattice of subsets with set inclusion as the partial order because all the $w'_i$ values are strictly positive. By Sperner's Theorem~\cite{sperner1928satz}, the size of this collection is at most ${k\choose \lfloor k/2\rfloor}$.
\end{proof}

\begin{proof}[Proof of Theorem \ref{thm::exp-Jeroslow-ineq-II}] We consider the instance from Theorem~\ref{thm::exp-Jeroslow-ineq-II}. For any split disjunction $D:= \{x:\langle a, x\rangle \leq b \} \cup \{x: \langle a, x \rangle \geq b+1\}$, we define $V(D)$ to be the set of all the optimal LP vertices (of the original polytope) that lie strictly in the corresponding split set $\{x: b \leq \langle a, x\rangle \leq b +1 \}$. Let the support of $a$ be given by $T \subseteq \{1, \ldots, n\}$ with $t:=|T| \leq s \leq \lfloor n/2\rfloor$. Since $a\in \Z^n$ and $b\in \Z$,  $V(D)$ is precisely the subset of the optimal LP vertices $\hat x$ such that $\langle a, \hat x \rangle = b + \frac12$. Fix some $\ell \in T$ and consider those optimal LP vertices $\hat x \in V(D)$ where $\hat x_\ell = \frac12$. This means that $\sum_{j \in T\setminus\{\ell\}}a_j \hat x_j= b+\frac12-\frac{a_\ell}{2}$. Let $r_i$ be the number of 0/1 solutions to $\sum_{j \in T\setminus\{\ell\}}a_j \hat x_j= b+\frac12-\frac{a_\ell}{2}$ with exactly $i$ coordinates set to 1. Then the number of vertices from $V(D)$ with the $\ell$-th coordinate equal to $\frac12$ is
$$\sum_{i=0}^{t-1} r_i {n-t\choose \lfloor n/2\rfloor - i} \leq \left(\sum_{i=0}^{t-1} r_i\right){n-t\choose \lfloor n/2\rfloor - \lfloor t/2\rfloor}.$$
since ${n-t\choose \lfloor n/2\rfloor - i} \leq {n-t\choose \lfloor n/2\rfloor - \lfloor t/2\rfloor}$ for all $i \in \{0, \ldots,t-1\}$. Using Lemma~\ref{lem:sperner}, $\sum_{i=0}^{t-1} r_i\leq {t-1\choose \lfloor t/2\rfloor}$ and we obtain the upper bound ${t-1\choose \lfloor t/2\rfloor}{n-t\choose \lfloor n/2\rfloor - \lfloor t/2\rfloor}$ on the number of vertices from $V(D)$ with the $\ell$-th coordinate equal to $\frac12$. Therefore, $|V(D)| \leq t{t-1\choose \lfloor t/2\rfloor}{n-t\choose \lfloor n/2\rfloor - \lfloor t/2\rfloor}=:p(t).$
Since $n$ is odd, we have
$$p(t)=
\begin{cases}
\displaystyle\frac{t!(n-t)!}{(t/2)!(t/2-1)!((n-t-1)/2)!((n-t+1)/2)!} &\mbox{if $t$ is even},\\[3mm]
\displaystyle\frac{t!(n-t)!}{((t-1)/2)!((t-1)/2)!((n-t)/2)!((n-t)/2)!} &\mbox{if $t$ is odd}.
\end{cases}$$
A direct calculation then shows that
$$\frac{p(t+1)}{p(t)}=
\begin{cases}
\displaystyle\frac{(t+1)(n-t+1)}{t(n-t)}&\mbox{if $t$ is even},\\
\displaystyle1&\mbox{if $t$ is odd}.
\end{cases}$$

Let $h$ be the largest even number not exceeding $s$.
Since $p(1)={n-1\choose \lfloor n/2\rfloor}$, we obtain, for every $t\in\{1,\dots,s\}$,
$$p(t)\le p(s)={n-1\choose \lfloor n/2\rfloor}\prod_{\substack{1\le q\le s \\ \mbox{$q$ even}}}\frac{q+1}q\cdot\frac{n-q+1}{n-q} = {n-1\choose \lfloor n/2\rfloor} \cdot \frac{(h+1)!!}{h!!}\cdot \frac{(n-1)!!}{(n-2)!!} \cdot \frac{(n-h-2)!!}{(n-h-1)!!},$$ where $m!!$ denotes the product of all integers from $1$ up to $m$ of the same parity as $m$. Using the fact that, for every even positive integer $\ell$,
$$\sqrt{\frac{\pi \ell}2}<\frac{\ell!!}{(\ell-1)!!}<\sqrt{\frac{\pi(\ell+1)}2}$$
(see, e.g., \cite{watson1959note,chen2005completely}), we have (for $h\ge1$, i.e., $s\ge2$)
\[\begin{split}
p(t)&\le{n-1\choose \lfloor n/2\rfloor}\cdot \frac{(h+1)(h-1)!!}{h!!} \cdot \frac{(n-1)!!}{(n-2)!!} \cdot \frac{(n-h-2)!!}{(n-h-1)!!}\\
&\le {n-1\choose \lfloor n/2\rfloor} (h+1)\sqrt{\frac2{\pi h}\cdot\frac{\pi n}2 \cdot \frac2{\pi (n-h-1)}}\\
&={n-1\choose \lfloor n/2\rfloor}\sqrt{\frac{2n(h+1)^2}{\pi h(n-h-1)}}\\
&={n-1\choose \lfloor n/2\rfloor}O\left(\sqrt{\frac{ns}{n-s}}\right).
\end{split}\]

Thus, this is an upper bound on $|V(D)|$.
Since the total number of optimal LP vertices of the instance is ${n{n-1 \choose \lfloor n/2 \rfloor}}$, we obtain the following lower bound of  on the size of a branch-and-bound proof: $\frac{{n{n-1 \choose \lfloor n/2 \rfloor}}}{|V(D)|}=\Omega\left(\sqrt{\frac{n(n-s)}{s}}\right).$ 
\end{proof}

\subsection{Proof of Theorem~\ref{thm:BB-vs-CP-sparsity}}

\begin{proof}[Proof of Theorem~\ref{thm:BB-vs-CP-sparsity}]
We first show a branch-and-bound algorithm with size $O(n)$.
Let the root node be $N_0$. The objective LP value of $N_0$ is $\frac{n}{2}$. Let $N_{1}^0$ and $N_{1}^1$ be the children of $N_0$  produced by branches $x_1\leq 0$ and $x_1\geq 1$ respectively. Then the LP values of $N_{1}^0$ and $N_{1}^1$ are $\floor{\frac{n}{2}}$ and $\frac{n}{2}$. Therefore $N_{1}^0$ is a leaf node. Recursively, let $N_{j+1}^0$ and $N_{j+1}^1$ be children of $N_{j}^1$ produced by $x_{j+1}\leq 0$ and $x_{j+1}\geq 1$ for $1\leq j \leq \floor{\frac{n}{2}}$. Note that this is well defined since the LP values of $N_{j}^0$ and $N_{j}^1$ are $\floor{\frac{n}{2}}$ and $\frac{n}{2}$ for $1\leq j \leq \floor{\frac{n}{2}}$. It is clear that node $N_{j+1}^0$ is a leaf for $1\leq j \leq \floor{\frac{n}{2}}$.  Node $N_{\ceil{\frac{n}{2}}}^1$ is an infeasible leaf since there are $\ceil{\frac{n}{2}}$ variables set to be $1$. Therefore, the whole branch-and-bound tree has $n+2$ nodes.

Next, we show that any cutting plane for the problem with sparsity $s\leq \floor{\frac{n}{2}}$ is valid for $[0,1]^n$. We will use the fact that $H\cap \{0,1\}^n=\{(x_1,x_2,\ldots,x_n)\in \{0,1\}^n:\sum_{i=1}^nx_i\leq \floor{\frac{n}{2}}\}$.

Let $S \subseteq \{1, \ldots, n\}$ be the set of indices for the non-zero coefficients in an inequality defining the cutting plane, i.e., the inequality is given by $\sum_{i\in S}a_ix_i\leq \delta$. Since this is a cutting plane it must be valid for all points in $H\cap \{0,1\}^n$. Let $V_S=\{(x_1,x_2,\ldots,x_n)\in \{0,1\}^n:x_i=0, i \not\in S\}$. Since $|S| \leq s\leq \floor{\frac{n}{2}}$, we have $V_S\subseteq H\cap \{0,1\}^n$. Therefore $\sum_{i\in S}a_ix_i\leq \delta$ is valid for all of $V_S$. Since the inequality only involves $x_i$, $i\in S$, it must also be a valid inequality for all of $\{0,1\}^n$.
\end{proof}

\subsection{Proof of Theorem~\ref{thm::general-BB>=CP}}

\begin{proof}[Proof of Theorem~\ref{thm::general-BB>=CP}]
Let the cutting plane proof be $H_1,H_2,\ldots, H_L$, and the sequence of the corresponding disjunctions deriving it be $D_1,D_2,\ldots, D_L \in\D$. Moreover, assume $H_i$ is $\langle \alpha_i,x\rangle\leq \delta_i$ for $1\leq i\leq L$. Since we assume all cutting planes are rational, we may assume $\alpha_i \in \Z^{n+d}$ and $\delta_i\in \Z$. Let $H'_i$ be $\langle \alpha_i,x\rangle\geq \delta_i+1$. Since $H_i$ is valid for $C \cap D_i$, we must have that $(C \cap H'_i) \cap D_i = \emptyset$.

Let $N_0 = C$ be the root node of the branch-and-bound tree. Recursively, we define $N_i$ and $N_i'$ be the children of $N_{i-1}$ generated by applying the split disjunction $H_{i}\cup H'_{i}$ for $1\leq i\leq L$. Applying the disjunction $D_{i}$ on $N_{i}'$ only generates infeasible nodes as noted above. Meanwhile, $N_i$ shows the validity of $H_i$. Thus, we have replaced the cut $H_i$ with $k+1$ nodes of the branch-and-bound tree: $k$ of these are infeasible and one is feasible. Therefore, we get a branch-and-bound tree of size $(k+1)L$.

A well-known family of instances in $\R^3$, given by $\conv\{(0,0,0),(2,0,0),(0,2,0), (\frac12,\frac12,h)\}$ for $h\in \N$, from~\cite{cook-kannan-schrijver} can be solved by branch-and-bound in $O(1)$ iterations with just variable disjunctions; however, there is a $\poly(\log(h))$ lower bound on the split rank \cite{conforti2015reverse}, and therefore, on the length of proofs based on split cuts.
\end{proof}

\subsection{Proofs of Theorems~\ref{thm:BC>BB/CP} and~\ref{thm:conv-BC}}

We will need some preliminary facts for comparing growth rate of instance sizes.

\begin{definition}\label{def:poly-equiv} A sequence of real numbers $(a_n)_{n\in \N}$ is said to {\em (asymptotically) polynomially dominate} another sequence $(b_n)_{n\in \N}$ if there exists a polynomial $p$, and two natural numbers $n_1, n_2\in \N$ such that $$\lim_{n\to \infty} \frac{b_{n_1 + n}}{p(a_{n_2+n})} < \infty.$$ If $(a_n)_{n\in \N}$ polynomially dominates $(b_n)_{n\in \N}$ and vice versa, we say that the two sequences are {\em (asymptotically) polynomially equivalent}.
\end{definition}

Note that if $b_n = O(p(a_n))$ for some polynomial $p$, then $(a_n)_{n\in \N}$ polynomially dominates $(b_n)_{n\in \N}$ (for example, $a_n = n$ is polynomially equivalent to the sequence $b_n = n^3$). However, our definition allows us to neglect a finite number of terms from both sequences. To illustrate the difference, consider the following two sequences. Define $a_1 = 2$,  and recursively $a_{n+1} = 2^{a_n}$ for $n \geq 2$. Define $b_n = a_{n+1}$ for $n \geq1$. There is no polynomial $p$ such that $b_n = O(p(a_n))$. Nevertheless, the sequence $(b_n)_{n\in \N}$ is simply a ``shift" of the sequence $(a_n)_{n\in \N}$ and we would like to say that both have the same growth rate. Our definition captures this situation.

The following two lemmas are direct consequences of Definition~\ref{def:poly-equiv}.

\begin{lemma}\label{lem:trivial-domination} Let $(a_n)_{n\in \N}$ and $(b_n)_{n\in \N}$ be two sequences such that $a_n \geq b_n$ for all $n \in \N$. Then $(a_n)_{n\in \N}$ polynomially dominates $(b_n)_{n\in \N}$.
\end{lemma}


\begin{lemma}\label{lem::interlaced} Let $(a_n)_{n\in \N}$ and $(b_n)_{n\in \N}$ be two sequences such that $a_n \leq b_n \leq a_{n+1}$ for all $n \in \N$. Then $(a_n)_{n\in \N}$ and $(b_n)_{n\in \N}$ are polynomially equivalent.
\end{lemma}

\begin{prop}\label{prop:equiv-subseq} Let $(a_n)_{n\in \N}$ and $(b_n)_{n\in \N}$ be two sequences such that $\lim_{n\to \infty} a_n = \infty = \lim_{n\to \infty} b_n$. Then there exist subsequences $(a'_n)_{n\in \N}$ and $(b'_n)_{n\in \N}$ of $(a_n)_{n\in \N}$ and $(b_n)_{n\in \N}$ respectively such that $(a'_n)_{n\in \N}$ and $(b'_n)_{n\in \N}$ are polynomially equivalent.
\end{prop}

\begin{proof} 
Since $\lim_{n\to \infty} a_n = \infty = \lim_{n\to \infty} b_n$, there exist subsequences $(a'_n)_{n\in \N}$ and $(b'_n)_{n\in \N}$ of $(a_n)_{n\in \N}$ and $(b_n)_{n\in \N}$ respectively such that $a_n \leq b_n \leq a_{n+1}$ for all $n \in \N$. Indeed, one can build this sequence inductively: Start with $a'_1 = a_1$, define $b'_1$ to be the smallest number in the sequence $(b_n)_{n\in \N}$ larger than or equal to $a'_1$. Suppose we have built up the subsequence upto some $i\in \N$: $a'_1, \ldots, a'_i$ and $b'_1, \ldots, b'_i$ such that $a'_k \leq b'_k \leq a'_{k+1}$ for all $k \leq i-1$ and $a'_i \leq b'_i$. Define $a'_{i+1}$ to be the smallest number in the sequence $(a_n)_{n\in \N}$ larger than or equal to $b'_i$, and define $b'_{i+1}$ to be the smallest number in the sequence $(b_n)_{n\in \N}$ larger than or equal to $a'_{i+1}$. By Lemma~\ref{lem::interlaced}, these two subsequences are polynomially equivalent.
\end{proof}

We next derive some straightforward consequences of Definition~\ref{def:regular}.

\begin{lemma}\label{lem:inclusion-BB}\label{lem:inclusion-CP}  Let $C \subseteq C'$ be two closed, convex sets. Let $\D$ be any branching scheme and let $\CP$ be an inclusion closed cutting plane paradigm. If there is a branch-and-bound proof with respect to $C'$ based on $\D$ for the validity of an inequality $\langle c, x \rangle \leq \gamma$, then there is a branch-and-bound proof with respect to $C$ based on $\D$ for the validity of $\langle c, x \rangle \leq \gamma$ of the same size. The same holds for cutting plane proofs based on $\CP$.\end{lemma}

\begin{proof} For the branch-and-bound proofs, apply the same set of disjunctions on $C$ instead of $C'$. Since $C\subseteq C'$, all the nodes in the branch-and-bound tree for $C$ are subsets of the corresponding nodes in the branch-and-bound tree for $C'$. Thus, $\langle c, x \rangle \leq d$ is valid for the leaves of the new branch-and-bound tree. 

For the cutting plane proofs, apply the same sequence of cuts and the result follows from the inclusion closed property of $\CP$ (Definition~\ref{def:regular}).\end{proof}

\begin{lemma}\label{lem:embedding-BB}\label{lem:embedding-CP} Let $\D$ and $\CP$ be both embedding closed and let $C\subseteq \R^{n_1} \times \R^{d_1}$ be a closed, convex set. Let $\langle c, x \rangle \leq \gamma$ be a valid inequality for $C \cap (\Z^{n_1}\times \R^{d_1})$. If there is a branch-and-bound proof with respect to $C \times \{0\}^{n_2} \times \{0\}^{d_2}$ based on $\D$ for the validity of $\langle c, x \rangle \leq \gamma$ interpreted as a valid inequality in $\R^{n_1} \times \R^{d_1} \times \R^{n_2} \times \R^{d_2}$ for $(C \times \{0\}^{n_2} \times \{0\}^{d_2}) \cap (\Z^{n_1}\times \R^{d_1}\times \Z^{n_2}  \times \R^{d_2})$, then there is a branch-and-bound proof with respect to $C$ based on $\D$ for the validity of $\langle c, x \rangle \leq \gamma$ of the same size. The same holds for cutting plane proofs based on $\CP$.
\end{lemma}

\begin{proof} Since $\D$ is embedding closed, for any disjunction $D$ used in the space $\R^{n_1} \times \R^{n_2} \times \R^{d_1} \times \R^{d_2}$, we use the restriction of $D$ to the space $\R^{n_1}\times \R^{d_1}$ (Definition~\ref{def:regular}).

Similarly, the cutting plane claim from the fact that $\CP$ is embedding closed (Definition~\ref{def:regular}).
\end{proof}

\begin{lemma}\label{lem:regular-CP-BB} Let $C\subseteq \R^{n+d}$ be a polytope and let $\langle c, x \rangle \leq \gamma$ be a valid inequality for $C \cap (\Z^n \times \R^d)$. Let $X := \{ (x,t) \in \R^{n+d} \times \R : x \in C,\;\; t = \langle c,x\rangle\}$. Then, for any regular branching scheme $\D$ or a regular cutting plane paradigm $\CP$, any proof of validity of $\langle c, x \rangle \leq \gamma$ with respect to $C \cap (\Z^n \times \R^d)$ can be changed into a proof of validity of $t \leq \gamma$ with respect to $X \cap (\Z^n \times \R^d \times \R)$ with no change in length, and vice versa.
\end{lemma}

\begin{proof} A proof of $\langle c, x \rangle \leq \gamma$ with respect to $C \cap (\Z^n \times \R^d)$ never involves $t$, and so can be carried over verbatim a proof for $t = \langle c, x \rangle \leq \gamma$ with respect to $X \cap (\Z^n \times \R^d \times \R)$. In the other direction, since we assume $\D$ is regular (Definition~\ref{def:regular}), no disjunction uses the variable $t$ and so it can be applied with the same effect on $C$. Similarly, since $\CP$ is regular, by definition any cutting plane derived for $X$ can be converted into an equivalent cutting plane for $C$.\end{proof}

\begin{proof}[Proof of Theorem~\ref{thm:BC>BB/CP}] Let $\{P_k \subseteq \R^{n_k} \times \R^{d_k} : k\in \N\}$ be a family of closed, convex sets, and $\{(c_k, \gamma_k) \in \R^{n_k} \times \R^{d_k}\times \R : k \in \N\}$ be a family of tuples such that $\langle c_k, x \rangle \leq \gamma_k$ is valid for $P_k \cap (\Z^{n_k} \times \R^{d_k})$, and $\CP$ has polynomial size proofs for this family of instances, whereas $\D$ has exponential size proofs. Similarly, let $\{P'_k \subseteq \R^{n'_k} \times \R^{d'_k} : k\in \N\}$ be a family of closed, convex sets, and $\{(c'_k, \gamma'_k) \in \R^{n'_k} \times \R^{d'_k}\times \R : k \in \N\}$ be a family of tuples such that $\langle c'_k, x \rangle \leq \gamma'_k$ is valid for $P'_k \cap (\Z^{n'_k} \times \R^{d'_k})$, and $\D$ has polynomial size proofs for this family of instances, whereas $\CP$ has exponential size proofs. By Proposition~\ref{prop:equiv-subseq}, we may assume that the sequence of sizes of the instances $(P_k,c_k, \gamma_k)$ and $(P'_k,c'_k, \gamma'_k)$ in the two families are polynomially equivalent, by passing to an infinite subfamily if necessary. Since the polynomial or exponential behaviour of the proof sizes are defined with respect to the sizes of the instances, passing to infinite subfamilies maintains this behaviour. 

We first embed $P_k$ and $P'_k$ into a common ambient space for each $k\in \N$. This is done by defining $\bar n_k = \max\{n_k, n'_k\}$, $\bar d_k = \max\{d_k, d'_k\}$, and embedding both $P_k$ and $P'_k$ into the space $\R^{\bar n_k} \times \R^{\bar d_k}$ by defining $Q_k := P_k \times \{0\}^{\bar n_k - n_k} \times \{0\}^{\bar d_k - d_k}$ and $Q'_k := P'_k \times \{0\}^{\bar n_k - n'_k} \times \{0\}^{\bar d_k - d'_k}$. By Lemma~\ref{lem:embedding-BB}, $\D$ has an exponential lower bound on sizes of proofs for the inequality $\langle c_k, x \rangle \leq \gamma_k$, interpreted as an inequality in $\R^{\bar n_k} \times \R^{\bar d_k}$, valid for $Q_k \cap (\Z^{\bar n_k} \times \R^{\bar d_k})$. By Lemma~\ref{lem:embedding-CP},  $\CP$ has an exponential  lower bound on sizes of proofs for the inequality $\langle c'_k, x \rangle \leq \gamma'_k$, interpreted as an inequality in $\R^{\bar n_k} \times \R^{\bar d_k}$, valid for $Q'_k \cap (\Z^{\bar n_k} \times \R^{\bar d_k})$.

We now make the objective vector common for both families of instances. Define $X_k := \{(x, t) \in \R^{\bar n_k}\times \R^{\bar d_k} \times \R: x \in Q_k, \;\; t = \langle c_k, x\rangle \}$ and $X'_k := \{(x, t) \in \R^{\bar n_k}\times \R^{\bar d_k} \times \R: x \in Q'_k, \;\; t = \langle c'_k, x\rangle \}$. By Lemma~\ref{lem:regular-CP-BB}, the inequality $t \leq \gamma_k$ has an exponential lower bound on sizes of proofs based on $\D$ for $X_k$ and the inequality $t \leq \gamma'_k$ has an exponential lower bound on sizes of proofs based on $\CP$ for $X'_k$.

We next embed these families as faces of the same closed convex set. Define $Z_k \subseteq \R^{\bar n_k} \times \R^{\bar d_k} \times \R \times \R$, for every $k\in \N$, as the convex hull of $X_k \times \{0\}$ and $X'_k \times \{1\}$. 

The key point to note is that these constructions combine two families whose sizes are polynomially equivalent and therefore the new family that is created has sizes that are polynomially equivalent to the original two families.

We let $(x,t,y)$ denote points in the new space $\R^{\bar n_k} \times \R^{\bar d_k} \times \R \times \R$, i.e., $y$ denotes the last coordinate. Consider the family of inequalities $t - \gamma_k(1-y) - \gamma'_k y \leq 0$ for every $k \in \N$. Note that this inequality reduces to $t \leq \gamma_k$ when $y=0$ and it reduces to $t \leq \gamma'_k$ when $y=1$. Thus, the inequality is valid for $Z_k \cap (\Z^{\bar n_k}\times \R^{\bar d_k} \times \R \times \Z)$, i.e., when we constrain $y$ to be an integer variable. Since $X_k \times \{0\} \subseteq Z_k$, by Lemma~\ref{lem:inclusion-BB}, proofs of $t - \gamma_k(1-y) - \gamma'_k y \leq 0$ based on $\D$ have an exponential lower bound on their size. Similarly, since $X'_k \times \{1\} \subseteq Z_k$, by Lemma~\ref{lem:inclusion-CP}, proofs of $t - \gamma_k(1-y) - \gamma'_k y \leq 0$ based on $\CP$ have an exponential lower bound on their size.

However, for branch-and-cut based on $\CP$ and $\D$, we can first branch on the variable $y$ (recall from the hypothesis that $\D$ allows branching on any integer variable). Since $\CP$ has a polynomial proof for $P_k$ and $(c_k, \gamma_k)$ and therefore for the valid inequality $t \leq \gamma_k$ for $X_k \times \{0\}$, we can process the $y=0$ branch with polynomial size cutting plane proofs. Similarly, $\D$ has a polynomial proof for $P'_k$ and $(c'_k, \gamma'_k)$ and therefore for the valid inequality $t \leq \gamma'_k$ for $X'_k \times \{1\}$, we can process the $y=1$ branch also in with polynomial size proofs. Thus, branch-and-cut gives polynomial size proofs overall for this family of instances. 
\end{proof}

\begin{proof}[Proof of Theorem~\ref{thm:conv-BC}] Recall that we restrict ourselves to the pure integer case, i.e., $d=0$. Consider any branch-and-cut proof for some instance. If no cutting planes are used in the proof, this is a pure branch-and-bound proof and we are done. Otherwise, let $N$ be a node of the proof tree where a cutting plane $\langle a, x \rangle \leq \gamma$ is used. Since we assume all cutting planes are rational, we may assume $a \in \Z^{n}$ and $\gamma\in \Z$. Thus, $N' = N \cap\{x : \langle a, x \rangle \geq \gamma + 1\}$ is integer infeasible. Since $\langle a, x \rangle \leq \gamma$ is in $\CP(N)$, by our assumption, there must be a branch-and-bound proof of polynomial size based on $\D$ for the validity of $\langle a, x \rangle \leq \gamma$ with respect to $N$. Since $N' \subseteq N$, by Lemma~\ref{lem:inclusion-BB}, there must be a branch-and-bound proof for the validity of $\langle a, x \rangle \leq \gamma$ with respect to $N'$, thus proving the infeasibility of $N'$. In the branch-and-cut proof, one can replace the child of $N$ by first applying the disjunction $\{x: \langle a, x \rangle \leq \gamma\} \cup \{x : \langle a, x \rangle \geq \gamma + 1\}$ on $N$, and then on $N'$, applying the above branch-and-bound proof of infeasibility. We now have a branch-and-cut proof for the original instance with one less cutting plane node. We can repeat this for all nodes where a cutting plane is added and convert the entire branch-and-cut tree into a pure branch-and-bound tree with at most a polynomial blow up in size.\end{proof}

\subsection*{Acknowledgments}
Amitabh Basu and Hongyi Jiang gratefully acknowledge support from ONR Grant N000141812096, NSF Grant CCF2006587, and AFOSR Grant FA95502010341. Michele Conforti and Marco Di Summa were supported by a SID grant of the University of Padova.

\bibliographystyle{plain}
\bibliography{full-bib}

\end{document}